\documentclass[11pt]{amsart}
\usepackage{amssymb, amsmath}
\title{On $\mathfrak{X}$-transitive groups and conjugate separable $\mathfrak{X}$-subgroups}
\author{O. Al-Raisi}
\address{O. Al-Raisi: Department of  Mathematics,  College of Science, Sultan Qaboos University, Muscat, Oman}

\author{M. Shahryari}
\address{M. Shahryari: Department of  Mathematics,  College of Science, Sultan Qaboos University, Muscat, Oman}
\email{m.ghalehlar@squ.edu.om}

\newtheorem{corollary}{Corollary}[section]
\newtheorem{proposition}{Proposition}[section]

\newtheorem {theorem}{Theorem}[section]
\newtheorem{lemma}{Lemma}[section]

\newtheorem{definition}{Definition}[section]

\numberwithin{equation}{section}

\newcommand{\XT}{\mathfrak{X}\mathrm{T}}
\newcommand{\CSX}{\mathrm{CS}\mathfrak{X}}
\newcommand{\X}{\mathfrak{X}}
\newcommand{\idX}{\mathrm{id}(\mathfrak{X})}
\newcommand{\idnX}{\mathrm{id}_n(\mathfrak{X})}

\begin{document}

\maketitle
\begin{abstract}
For a given variety of groups $\X$, we develop a systematic theory of $\CSX$-groups and $\XT$-groups, extending ideas proposed in \cite{Shah}. We analyze the interplay between these classes, describe their structural properties, and examine their connections with equational domains and residually $A$-free groups. Furthermore, we prove by elementary means that every finite $\CSX$-group lies in $\X$.
\end{abstract}
\vspace{1cm}

{\bf AMS Subject Classification} 20E06, 20E26, 20F70.\\
{\bf Keywords} $\mathrm{CT}$-group; $\mathrm{CSA}$-group; conjugately separable subgroup; variety of groups; inductive class; universal class; residually free group; residually $A$-free group.

\vspace{1cm}


A subgroup $H$ of a given group $G$ is called {\em malnormal} if for every element $x\in G\setminus H$ we have  $H\cap H^x=\{1\}$. A group $G$ is called {\em conjugate separable abelian} ($\mathrm{CSA}$) if all maximal abelian subgroups of $G$ are malnormal.
The class of $\mathrm{CSA}$-groups is broad and plays a significant role in the study of residually free groups, universal theory of non-abelian free groups, limit groups, exponential groups and equational domains in algebraic geometry over groups (see \cite{BMR}, \cite{Champ}, \cite{MR1} , and \cite{MR2}). Another class of groups with very close connections to $\mathrm{CSA}$-groups is the class of $\mathrm{CT}$-groups ({\em commutative transitive groups}). A group is $\mathrm{CT}$ if commutativity is a transitive relation on the set of its non-identity elements. Despite this simple definition, the class of CT groups is also central to the study of residually free groups. Every $\mathrm{CSA}$-group is $\mathrm{CT}$ but the converse is not true. B. Baumslag proved that in the presence of residual freeness, both properties are equivalent.

In recent decades, considerable effort has been devoted to the study of these classes and their generalizations. A generalization of $\mathrm{CT}$-groups is introduced in \cite{Ciob} to extend the above mentioned theorem of B. Baumslag. Many interesting relations between $\mathrm{CSA}$- and $\mathrm{CT}$-groups are presented in \cite{Fine}, along with an extensive review of the existing literature.

The second author of the present paper proposed an idea in \cite{Shah} for generalizing the notions of $\mathrm{CSA}$- and $\mathrm{CT}$-groups: Suppose that $\X$ is a class  of  groups. A group $G$  is an   $\XT$-group  if and only if, for any two $\X$-subgroups $A$ and $B$ of $G$, the assumption $A\cap B\neq \{1\}$ implies that $\langle A, B\rangle$ is also an $\X$-subgroup of $G$. Similarly,  a group $G$ is a $\CSX$-group if  all of its maximal $\X$-subgroups  are malnormal. In \cite{Shah}, he examined the special case where $\X=\mathfrak{N}_k$, the variety of nilpotent groups of nilpotency class not exceeding $k$, which contains $\mathrm{CSA}$- and $\mathrm{CT}$-groups as particular instances. Hence, \cite{Shah} addresses the study of general relations among the classes of $\CSX$- and $\XT$-groups, their characterizations, constructions, and universal axiomatization, as well as their connections to residual $A$-free groups, in the case when $\X$ is the variety of nilpotent groups of class at most $k$. Furthermore, many previous results are shown to remain valid for $\CSX$ and $\XT$-groups when $\X$ is the variety $\mathfrak{N}_k$. As an application, the ideas of \cite{Shah} are used in \cite{Omar-Shah} to introduce a large class of groups which are {\em equational domain} in the sense of algebraic geometry over groups (see \cite{BMR} for definitions).

We develop the general theory in this paper: For a given class $\X$ (which is taken to be a variety in most of the cases), we study the classes of $\CSX$- and $\XT$-groups, their relations, characterizations, and constructions. We show that some of the previous results extend to $\CSX$ and $\XT$-groups under some assumptions on the class $\X$. However, certain properties do not hold in this more general framework.

In the first section, we discuss the basic properties of the classes $\XT$ and $\CSX$. In the second section, we prove that finite $\CSX$-groups belong to $\X$ using a completely elementary argument that avoids any results from representation theory or the classification of finite simple groups. In Section 3 we discuss the inclusion $\CSX\subseteq \XT$ and, by means of a counterexample, we show that this inclusion does not hold in general. Section 4 shows that for inductive classes of groups containing all abelian groups, every $\CSX$-group which is not an $\X$-group is an equational domain. In Section 5 we show that the free product of two $\CSX$-groups belongs to $\CSX$. Then in the following section we show that for a large class of varieties, the classes $\XT$ and $\CSX$ are universal. Finally, in Section 7 we investigate the relations between the classes $\XT$, $\CSX$, and (fully) residually $A$-free groups.

For clarity, we introduce the following notation. The subgroup generated by a subset $X$ of a group $G$ will be denoted by $\langle X\rangle$, and the normal closure of this subgroup is $\langle X^G\rangle$. A conjugate $a^x$ (or $H^x$) is $x^{-1}ax$ (similarly, $x^{-1}Hx$). A commutator $[x, y]$ is $x^{-1}y^{-1}xy$ and all simple commutators $[x_1, x_2, \ldots, x_{k+1}]$ are left aligned. The notation $[x,_k y]$ stands for $[x, y, \ldots, y]$ where $y$ occurs $k$ times.  For a class $\{ G_i\}_{i\in I}$ of groups, the corresponding free product will be denoted by $\prod^{\ast}_{i\in I}G_i$. A variety generated by a single group $G$ (single identity $x\approx 1$) is denoted by $\mathrm{Var}(G)$ (likewise, $\mathrm{Var}(w\approx 1)$). In most of the cases, we prefer to denote a law $w(x_1, \ldots, x_n)\approx 1$ simply by the element $w$ of the corresponding free group. Hence, whenever we say that $w$ is an identity of the variety $\X$, we mean that all elements of $\X$ satisfy the law $w(x_1, \ldots, x_n)\approx 1$. The set of all $n$-variable identities of a variety $\X$ will be denoted by $\idnX$; we write $\idX$ for the set of all identities of $\X$.

\section{Basic Properties and Generalized Centralizers}
It is known that a group $G$ is $\mathrm{CT}$ if and only if, for every maximal abelian subgroup $A$ of $G$, the centralizer of every non-trivial element of $A$ is $A$ itself. Accordingly, it seems natural to generalize the notion of a centralizer of an element in a group when studying $\XT$- and $\CSX$-groups; we do so in this section.  In what follows, by an {\em inductive class} we mean a class $\X$ which is closed under subgroups and under the union of any chain of its members. Interestingly, every class with this property is also closed under direct limit of injective direct systems. The proof of this classical fact of lattice theory rests on a transfinite induction and the reader can find it in \cite{Mark} (there, it is shown that even if a class is closed under union of well-ordered chains, then it contains the limits of all injective direct systems). Consequently, a group belongs to an inductive class $\X$ if and only if, all of its finitely generated subgroups belong to $\X$. Obviously,  all {\em universal classes} (classes which are axiomatized by universal sentences in the first order language of groups) and, in particular, all varieties of groups, are inductive.

\begin{definition}
Let $\X$ be an inductive class of groups. Given a group $G$ and an element $a\in G$, we define the $\X$-centralizer of $a$ in $G$, denoted $C_{\X}^G(a)$ (or simply $C_{\X}(a)$ when $G$ is understood from the context)  by
$$
C_{\X}^G(a)=\{x\in G:\ \langle a,x \rangle\in \X \}.
$$
\end{definition}

A few remarks regarding this definition are in order. First, note that the $\X$-centralizer of $a$ will not be empty if $\X$ contains all cyclic groups since $1\in C_\X(a)$ in this case.  Moreover, it is clear that if $x\in C_{\X}(a)$ then so is $x^{-1}$. However, $C_{\X}(a)$ is not necessarily a subgroup of $G$ since the product of two elements in $C_{\X}(a)$ need not be an element of $C_{\X}(a)$. To see this, let $\X$ be the variety of metabelian groups, $G=S_4$, $a=(2\ 3)$, $x=(1\ 2)$ and $y=(2\ 3\ 4)$. Then, both $x$ and $y$ belong to $C_{\X}(a)$ but, $xy$ does not belong to $C_{\X}(a)$. Of course, when $\X$ is the variety of abelian groups, $C_{\X}(a)$ is simply the centralizer of $a$ in $G$ which is indeed a subgroup of $G$.

\begin{proposition}\label{Cent-1}
Let $\X$ be an inductive class of groups. Then, a group $G$ belongs to $\XT$ if and only if, for every non-identity element $a\in G$, the set $C_{\X}(a)$ is empty or an $\X$-subgroup of $G$.
\end{proposition}

\begin{proof}
Suppose $G\in \XT$ and $a$ is a non-identity element of $G$. Let $C_{\X}(a)$ be non-empty; we must show that it is closed under multiplication. So, let $x, y\in C_{\X}(a)$. This means that the subgroups $\langle a, x\rangle$ and $\langle a, y\rangle$ of $G$ belong to $\X$. Then, the assumption $G\in \XT$ implies that $\langle a, x, y\rangle \in \X$, and consequently, $\langle a, xy\rangle \in \X$. This shows that $C_{\X}(a)$ is a subgroup of $G$. To prove $C_{\X}(a)\in \X$, we must show that every finitely generated subgroup of $C_{\X}(a)$ belongs to $\X$. Suppose  $b_1, \ldots, b_n\in C_{\X}(a)$. Then, we have
$$
\langle a, b_1\rangle, \ldots, \langle a, b_n\rangle \in \X,
$$
and hence, by $\XT$, we must have $\langle a, b_1, \ldots, b_n\rangle\in \X$. This implies that $\langle b_1, \ldots, b_n\rangle\in \X$ and therefore, $C_{\X}(a)\in \X$.

Conversely, suppose that, for every non-identity element $a\in G$, $C_{\X}(a)$ is empty or an $\X$-subgroup of $G$. Let $A$ and $B$ be $\X$-subgroups of $G$ and assume that  $A\cap B$ contains a non-identity element $a$. We will show that $\langle A,B\rangle$ is an $\X$-subgroup of $G$. To this end, note that $\langle a, x\rangle\in \X$ for each $x\in A$ since both $a$ and $x$ are in $A$ which is an $\X$-group. This shows that $A\subseteq C_{\X}(a)$. A similar argument shows that $B\subseteq C_{\X}(a)$. But we are assuming that $C_{\X}(a)$ is an $\X$-subgroup of $G$ so, if both $A$ and $B$ are subgroups of $C_{\X}(a)$, then so is $\langle A,B\rangle$. In particular, $\langle A,B\rangle$ is an $\X$-subgroup of $G$, proving that $G$ is an $\XT$-group.
\end{proof}

\begin{lemma}
Let $\X$ be a subgroup-closed class of groups. Let $G$ be a group, and  $a$ be a non-identity element of $G$. If $C_{\X}(a)$ is an $\X$-subgroup of $G$, then it is a maximal $\X$-subgroup of $G$.
\end{lemma}

\begin{proof}
Let $a$ be a non-identity element of $G$ and suppose that $C_{\X}(a)$ is an $\X$-subgroup of $G$. Let $H$ be an $\X$-subgroup of $G$ that contains $C_{\X}(a)$ and note that, if $h$ is an element of $H$, then $\langle a, h \rangle$ is a subgroup of the $\X$-group $H$, and so, $\langle a, h \rangle\in\X$. But this means that $h\in C_{\X}(a)$ which shows that $H=C_{\X}(a)$.
\end{proof}

\begin{lemma}
Let $\X$ be a subgroup-closed class of groups and $G$ be an $\XT$-group. Suppose  $M$ is a maximal $\X$-subgroup of $G$. Then, for every $a\in M\setminus\{1\}$ we have $M=C_{\X}(a)$.
\end{lemma}

\begin{proof}
Let $a\in M\setminus \{1\}$. Then, for every $x\in M$, we have $\langle a, x\rangle\subseteq M$, and hence, $x\in C_{\X}(a)$. Thus, $M\subseteq C_{\X}(a)$ and, by the maximality assumption on $M$, we have $M=C_{\X}(a)$.
\end{proof}

\begin{proposition}\label{Cent-2}
Let $\X$ be an inductive class of groups. Suppose that $\CSX\subseteq \XT$. Then, a group $G$ is a $\CSX$-group if and only if for every non-identity element $a$ of $G$, $C_{\X}(a)$ is empty or a malnomral $\X$-subgroup of $G$.
\end{proposition}

\begin{proof}
If $G$ is a $\CSX$-group, then it is an $\XT$-group by hypothesis, and accordingly, every non-empty $\X$-centralizer is a maximal $\X$-subgroup. As such, every non-empty $\X$-centralizer must be malnormal in $G$. Conversely,  if every non-empty $\X$-centralizer  is a malnormal $\X$-subgroup of $G$, then every maximal $\X$-subgroup of $G$ is malnormal in $G$ by the previous lemma.
\end{proof}

\begin{proposition}\label{CSX}
Let $\X$ be inductive and  $G$ be a group. If $G$ is a  $\CSX$-group, then for all $a\in G\setminus\{1\}$ and all $z\in G$, the condition $\langle a,a^z \rangle\in \X$ implies $\langle a,z \rangle\in \X$. Conversely, if $G$ is an  $\XT$-group and if for all $a\in G\setminus \{1\}$ and all $z\in G$, the condition $\langle a,a^z \rangle\in \X$ implies $\langle a,z \rangle\in \X$, then  $G$ is a $\CSX$-group.
\end{proposition}

\begin{proof}
Suppose that $G$ is a  $\CSX$-group; let $a\neq 1$ and $z$ be elements of $G$, and suppose that $\langle a,a^z \rangle\in \X$. Let $M$ be a maximal $\X$-subgroup of $G$ containing the subgroup $\langle a, a^z\rangle$ of $G$. We claim that $z\in M$. Indeed, $M$ is malnormal in $G$ and
$$
1\neq a\in M\cap M^{z^{-1}}.
$$
Hence, $z\in M$, and consequently,  $\langle a,z\rangle\in \X$.

Next, suppose that $G$ is an $\XT$-group and that, for all $a\in G\setminus\{ 1\}$ and all $z\in G$, the condition $\langle a,a^z \rangle\in \X$ implies $\langle a,z \rangle\in \X$. Let $M$ be a maximal $\X$-subgroup of $G$ and suppose that, for some $g\in G$, the intersection $M\cap M^g$ contains a non-trivial element $a$. Now, since $a\in M$ and $a\in M^g$, it follows that $\langle a, a^z\rangle \subseteq M$, where $z=g^{-1}$. But then, $\langle a, a^{z}\rangle$ is an $\X$-subgroup of $G$ as it is  a subgroup of the $\X$-subgroup $M$ of $G$. Therefore, $\langle a, z\rangle \in \X$. Using the $\XT$-property of $G$, we may conclude from $\langle a,z\rangle \in \X$ and $M\in \X$ that $\langle M,z\rangle\in \X$. However, $M$ is assumed to be a maximal $\X$-subgroup of $G$ so we must have $\langle M,z\rangle=M$ which implies that $z=g^{-1}$ belongs to $M$. This shows that $M$ is a malnormal subgroup of $G$ and concludes the proof.
\end{proof}

It can be easily verified that, for any class of groups $\X$, the class $\XT$ is subgroup-closed. We show that the class $\CSX$ has a similar property when $\X$ is inductive.

\begin{proposition}\label{Subgroup}
If $\X$ is an inductive class of groups, then the class of all $\CSX$-groups is subgroup-closed.
\end{proposition}

\begin{proof}
Let $G$ be a $CS\X$-group and $H$ be a subgroup of $G$; let $M$ be a maximal $\X$-subgroup of $H$ and $M'$ be a maximal $\X$-subgroup of $G$ containing $M$ (note that $M_0$ exists by Zorn's lemma since $\X$ is assumed to be inductive). Let $h\in H$ be such that $M\cap M^h$ is non-trivial. As $M\cap M^h\subseteq M_0\cap M_0^h$, it follows that $ M_0\cap M_0^h$ is not the trivial subgroup of $G$, and since $M_0$ is malnormal in $G$, we infer that $h\in M_0$. As such, $h\in M_0\cap H$. Now, $M\subseteq M_0\cap H$ and $M_0\cap H$ is an $\X$-subgroup of $H$ because the class $\X$ is subgroup-closed. But $M$ is a maximal $\X$-subgroup of $H$ so we must have $M=M_0\cap H$. Consequently,  $h\in M_0\cap H=M$.
\end{proof}

Special cases of the next result are established for $\mathrm{CT}$-groups (see \cite{Fine} and \cite{MR2}) and $\XT$-groups where $\X=\mathfrak{N}_k$ (see \cite{Shah}).

\begin{proposition}
Let $\X$ be an inductive class containing all abelian groups. If an $\XT$-group $G$ is decomposable, then $G$ belongs to $\X$
\end{proposition}

\begin{proof}
Let $G$ belong to $\XT$ and assume that $G=A\times B$ for non-trivial groups $A$ and $B$. Suppose  $a_1, \ldots, a_n\in A$, and $b_1, \ldots, b_m\in B$ are non-identity elements. We know that each of the subgroups $\langle a_i, b_1\rangle$ is abelian, and hence, belongs to $\X$. This shows that
$$
\langle a_1, \cdots, a_n, b_1\rangle \in \X,
$$
and repeating the same argument, we see that 
$$
\langle a_1, \cdots, a_n, b_1, \ldots, b_m\rangle \in \X.
$$
 Now, the assumption of being inductive implies that $A\times B\in \X$. 
\end{proof}

A similar argument can be used to verify the next result. Compare the last statement of this proposition with Lemma \ref{Marginal-2}.

\begin{proposition}
Let $\X$ be an inductive class  containing all abelian groups.
\begin{enumerate}
\item If the intersection of any two distinct maximal $\X$-subgroups of $G$ is trivial, then $G$ is an  $\XT$-group.
\item If $G$ is an $\XT$-group, then the intersection of any two distinct maximal $\X$-subgroups of $G$ is trivial.
\item If $G$ is a torsion-free $\XT$-group and $x^m=y^n$, for some $x,y\in G$ and some non-zero integers $n$ and $m$, then $\langle x,y\rangle\in \X$.
\item If an $\XT$-group $G$ does not belong to $\X$, then it is centerless, and hence, not nilpotent.
\end{enumerate}
\end{proposition}

The classes $\XT$ and $\CSX$ enjoy two further properties; they are inductive whenever $\X$ is inductive.

\begin{proposition}
If $\X$ is an inductive class of groups, then so is $\XT$.
\end{proposition}

\begin{proof}
Suppose that $\X$ is an inductive class of groups. We know that the class $\XT$ is subgroup-closed. It remains to show if $(G_i)_{i\in I}$ is an ascending chain of elements of $\XT$, then the union $G$ of this chain belongs $\XT$. Note that we may consider the index set $I$ to be a linearly ordered set so that $i\leq j$ implies $G_i\subseteq G_j$.  Let $A$ and $B$ be $\X$-subgroups of $G$ and suppose that $A\cap B$ contains a non-identity element $x$ of $G$. For each $i\in I$, let $A_i=A\cap G_i$ and $B_i=B\cap G_i$; note that each $A_i$ and each $B_i$ belongs to $\X$. Since $x\in G$ and $G$ is the union of $(G_i)_{i\in I}$, there must be an $i_0\in I$ such that $x\in G_{i_0}$; as such, $x\in A_{i_0}\cap B_{i_0}$. Therefore, for each $i\geq i_0$, the intersection $A_i\cap B_i$ is non-trivial and so, the join $H_i=\langle A_i, B_i\rangle$ is an $\X$-subgroup of $G$. Finally, it is easy to see that
$$
\langle A, B\rangle=\bigcup_{ i\geq i_0}H_i.
$$
Since the class $\X$ is inductive and $(H_i)_{i\in I}$ is an ascending chain, we conclude that $\langle A, B\rangle$ belongs to $\X$.
\end{proof}

\begin{proposition}
Let $\X$ be an inductive class of groups such that $\CSX\subseteq \XT$. Then, $\CSX$ is an inductive class of groups.
\end{proposition}

\begin{proof}
We already showed  that the class $\CSX$ is subgroup-closed assuming that the class $\X$ is inductive and $\CSX\subseteq \XT$; thus, it remains to show that $\CSX$ is closed under ascending chains. To this end, let $(G_i)_{i\in I}$ be an ascending chain of $\CSX$-groups and let $G$ be the union of this chain. Let $M$ be a maximal $\X$-subgroup of $G$ and suppose $M\cap M^g\neq \{1\}$ for some $g\in G$; consider a non-identity element $x=m^g\in M\cap M^g$, where $m\in M$. We know that there is an index $j\in I$ such that $g, x, m\in G_j$. Also, we have
$$
x\in (G_j\cap M)^g\cap(G_j\cap M).
$$
We prove that $G_j\cap M$ is a maximal $\X$-subgroup of $G_j$. Suppose $G_j\cap M\subseteq K\subseteq G_j$ where $K$ is an $\X$-subgroup. Recall from the previous proposition that $G$ belongs to $\XT$ and we have $M\cap K\neq \{1\}$ (as it contains $x$). Hence, $\langle M, K\rangle$ is an $\X$-subgroup of $G$, and therefore, by the maximality of $M$, we have $K\subseteq M$. This means that $K=G_j\cap K\subseteq G_j\cap M$. Hence, $G_j\cap M$ is a maximal $\X$-subgroup of $G_j$. As $G_j$ is $\CSX$, we must have $g\in G_j\cap M$, and hence, $g\in M$, proving that $M$ is malnormal in $G$.
\end{proof}

Although one might expect that iterating $\X\mapsto \XT$ produces infinitely many new classes, the following result shows that this does not occur when $\X$ is subgroup-closed.

\begin{proposition}
For every subgroup-closed class of groups $\X$, we have $\XT^2=\XT$, where  $\XT^2$ stands for the class $(\XT)T$.
\end{proposition}

\begin{proof}
The inclusion $\XT\subseteq \XT^2$ is trivially true. For the reverse inclusion, suppose that $G\in \XT^2$ and let $A$ and $B$ be $\X$-subgroups of $G$ with $A\cap B\neq \{1\}$. Since $\X$ is assumed to be subgroup-closed, both $A$ and $B$ belong to the class $\XT$.  Now, using the assumption that $G\in \XT^2$, we conclude that $H=\langle A, B \rangle\in \XT$. But both $A$ and $B$ are $\X$-subgroups of the $\XT$-group $H$ and $A\cap B\neq \{1\}$ so we must have $\langle A,B\rangle\in \X$. This shows that $G\in \XT$ and concludes the proof.
\end{proof}

A similar result is valid for the case of the assignment $\X\mapsto \CSX$ provided that stronger assumptions are imposed.

\begin{proposition}
For every inductive class of groups $\X$ satisfying $\CSX\subseteq \XT$, we have $\mathrm{CS}^2\X=\CSX$, where $\mathrm{CS}^2\X=\mathrm{CS}(\CSX)$.
\end{proposition}

\begin{proof}
Let $G\in \CSX$ and observe that the only maximal $\CSX$-subgroup of $G$ is $G$ itself which is malnormal. This establishes that $G\in \mathrm{CS}^2\X$. Next, suppose that $G\in \mathrm{CS}^2\X$ and let $M$ be a maximal $\X$-subgroup of $G$. Clearly, $\X\subseteq \CSX$ so $M$ is a $\CSX$-subgroup of $G$. Because the class $\X$ is inductive and $\CSX\subseteq \XT$, the class $\CSX$ is also inductive, as noted above. So, we may extend $M$ to a maximal $\CSX$-subgroup $M^{\prime}$ of $G$. Now, $M$ is a maximal $\X$-subgroup of $G$ that is contained in $M^{\prime}$ so $M$ must be a maximal $\X$-subgroup of the $\CSX$-group $M^{\prime}$; as such $M$ is malnormal in $M^{\prime}$. Furthermore, since $G$ is assumed to be a $\mathrm{CS}^2\X$-group, we may infer that $M^{\prime}$ is malnormal in $G$. Finally, it is easy to see that malnormality is transitive; that is, because $M$ is malnormal in $M^{\prime}$ and $M^{\prime}$ is malnormal in $G$, it follows at once that $M$ is malnormal in $G$. This shows that $G\in \CSX$ which concludes the proof.
\end{proof}

As noted above, most propositions concerning properties of $\CSX$-groups rely on the assumption $\CSX\subseteq \XT$. This inclusion is known to hold for the case where $\X$ is the variety of abelian groups as well as the case when $\X=\mathfrak{N}_k$. We shall present further examples of varieties with this property; however, counterexamples also exist as we shall see.

\section{Finite $\CSX$-Groups}

Finite $\mathrm{CSA}$ groups are abelian. All known proofs of this fact relay on representation theory, the classification of finite simple groups, or the Feit-Thompson's theorem. The reader may see \cite{Fine}, \cite{Ould}, or \cite{Marin}. In \cite{Shah} a  similar result is proved for the case of $\X=\mathfrak{N}_k$: A finite $\CSX$ group belongs to $\X$. The proof given in \cite{Shah} depends on the structure of finite Frobenius groups. In this section, we give a completely elementary proof for the next result.

\begin{theorem}\label{Finite}
Let $\X$ be a class of groups satisfying the following:
\begin{enumerate}
\item Every finite cyclic group is an $\X$-group.
\item Every finite  $\CSX$-group belongs to $\XT$.
\end{enumerate}
Then, every finite $\CSX$-group is an $\X$-group.
\end{theorem}

Before presenting the proof, it is worth mentioning that a more general statement is true. However, its proof employs properties of Frobenius groups (see \cite{Isaacs} for the structure of Frobenius groups).

\begin{proposition}
Let $\X$ be a class of groups which contains all finite cyclic groups. Then, every finite $\CSX$-group belongs to $\X$.
\end{proposition}

\begin{proof}
Let $G$ be a finite $\CSX$-group and suppose $G$ does not belong to $\X$. Consider two maximal $\X$-subgroups $A$ and $B$ of $G$. Note that as $\X$ contains all finite  cyclic groups, we may infer that $A$ and $B$ are non-trivial, and since $G$ does not belong to $\X$, we may conclude that $A$ and $B$ are proper subgroups of $G$. The malnormality of $A$ and $B$ implies that they are Frobenius complements in $G$, and hence, they are conjugate (see \cite{Isaacs}). This means that $B=A^x$ for some $x\in G$, and consequently, every maximal $\X$-subgroup of $G$ has this form. Now, for any element $g\in G$, the cyclic subgroup $\langle g\rangle$ must be contained in some maximal $\X$-subgroup which yields
$$
G=\bigcup_xA^x.
$$
This leads to the contradiction $G=A$.
\end{proof}

We now return to Theorem \ref{Finite} and present its proof.

\begin{proof}  Suppose that $G$ is a finite $\CSX$-group. The first step is to obtain a special covering of $G$ as follows. We begin by picking a non-trivial element $m_1$ of $G$. This element generates a non-trivial cyclic subgroup $\langle m_1\rangle$ of $G$ which is an $\X$-group by our assumption. Since $G$ is assumed to be finite, there is a maximal $\X$-subgroup $M_1$ of $G$ that contains this cyclic group. If this group  $M_1$ happens to be $G$, then $G$ is an $\X$-group and we are done. If not, then we consider the set $[M_1]$ of all distinct conjugates of $M_1$ in $G$ of which there are $[G:N_G(M_1)]$ many. Observe that, by malnormality of $M_1$ in $G$, $N_G(M_1)=M_1$ so that the size of $[M_1]$ is actually $[G:M_1]$. Now, every element in $[M_1]$ is a maximal $\X$-subgroup.  Furthermore, malnormality implies that for each $x\in G\setminus M_1$ the intersection of $M_1$ and $M_1^x$ is trivial. Moreover, if $M_1^x$ and $M_1^y$ are distinct elements of $[M_1]$, then $M_1^x\cap M_1^y=\{1\}$. Indeed, if $M_1^x\cap M_1^y\neq \{1\}$, then $M\cap M^{yx^{-1}}\neq \{1\}$; the malnormality of $M_1$ forces $M_1= M_1^{yx^{-1}}$ and so, $M_1^x=M_1^y$ which is a contradiction. Now, if the union $U_1$ of all the members in $[M_1]$ happens to coincide with $G$ then we obtained our desired covering of $G$. If not, then we pick an element $m_2$ in $G\setminus U_1$; this element generates a cyclic group which is contained in a maximal $\X$-subgroup $M_2$ of $G$ as previously argued. Note that since $M_2$ contains $m_2\in G\setminus U_1$, the $\XT$-property implies that no member of $[M_1]$ can intersect $M_2$ non-trivially. In fact, more can be said here; no element of the set $[M_2]$ (the set of all distinct conjugates of $M_2$) can intersect an element of $[M_1]$ non-trivially. Indeed, if $M_1^x\cap M_2^y\neq \{1\}$, then the $\XT$ property would imply $M_1^x= M_2^y$ which in turn yields a contradiction. Suppose $U_1\cup U_2=G$, where $U_2$ is the union of all members of $[M_2]$. Then, we have our desired covering of $G$; if not we proceed as above. Since $G$ is assumed to be finite, this process must eventually terminate resulting in distinct maximal $\X$-subgroups $M_1, \ldots, M_r$ of $G$ such that the following hold:
\begin{enumerate}
\item For each $i$ and $j$ with $i\neq j$, every member of the set $[M_i]$  intersects trivially with every member of the set $[M_j]$.
\item If $U_i$ is the union of all members of $[M_i]$, then
 $$
 G=\bigcup_{i=1}^{r}U_i.
 $$
 \item Each $M_i$ contains a non-identity element $m_i$.
 \end{enumerate}
With this special covering of $G$ at hand, we can count the number of elements in $G$ as follows:
\begin{eqnarray*}
|G|-1&=&\sum_{k=1}^{r}[G:N_G(M_k)](|M_k|-1)\\
     &=&\sum_{k=1}^{r}[G:M_k](|M_k|-1)\\
     &=&\sum_{k=1}^{r}\frac{|G|}{|M_k|}(|M_k|-1).
\end{eqnarray*}
Dividing both sides by $|G|$ yields
\begin{eqnarray*}
1-\frac{1}{|G|}&=&\sum_{k=1}^{r}(1-\frac{1}{|M_k|})\\
               &=&r-\sum_{k=1}^{r}\frac{1}{|M_k|}.
\end{eqnarray*}
Upon rearranging this we get the following
$$
r=1+\sum_{k=1}^{r}\frac{1}{|M_k|}-\frac{1}{|G|}<1+\sum_{k=1}^{r}\frac{1}{|M_k|}.
$$
Recall now that each $M_k$ contains a non-identity element so that, $1/|M_k|\leq 1/2$ for each $k$. Combining this with the above gives rise to the following inequalities
$$
r<1+\sum_{k=1}^{r}\frac{1}{|M_k|}\leq 1+\sum_{k=1}^{r}\frac{1}{2}=1+r/2.
$$
That is $r<2$, and hence, $r=1$. That is, $G$ is actually the union $U_1$ of all members of $[M_1]$. Of course, this implies that $G=M_1$ showing that $G$ belongs to $\X$.
\end{proof}

As an application, we now have an elementary proof of the following result as well.

\begin{corollary}
Let $\X$ be an inductive class of groups that contains all finite cyclic groups and suppose that $\CSX\subseteq \XT$. If a $\CSX$-group $G$ is locally finite, then it belongs to the class $\X$.
\end{corollary}

\begin{proof}
To show that a locally finite $\CSX$-group $G$ belongs to $\X$, it suffices to show that every finitely generated subgroup $H$ of $G$ belongs to $\X$. But, a finitely generated subgroup $H$ of $G$ is finite and is a $\CSX$-group by Proposition \ref{Subgroup}. This yields the desired result.
\end{proof}

A natural question is the classification of finite $\XT$-groups for various classes $\X$. This has already been accomplished for finite $\mathrm{CT}$ (see \cite{Fine} for a complete history). For other cases, however, the  classifications remain open, although it may be possible to obtain results for certain special varieties of groups.

\section{The Implication $\CSX\to \XT$}

As previously stated, the class of $\mathrm{CSA}$-groups is included in the class of $\mathrm{CT}$-groups; the same is true for the case when $\X=\mathfrak{N}_k$, the variety of nilpotent groups of class at most $k$ (see \cite{Shah}). In this section, we show that the same property is valid for any inductive class of nilpotent groups.  But, as emphasized before, there are varieties $\X$ for which the implication $\CSX\to \XT$ is false; an example of such a variety is provided below.

\begin{proposition}\label{Nilpotent}
Let $\X$ be an inductive class of groups such that every element of  $\X$ is nilpotent.  Then, $\CSX\subseteq \XT$.
\end{proposition}

\begin{proof}
Let $G$ be a $\CSX$-group. Suppose that $A$ and $B$ are $\X$-subgroups of $G$ with $A\cap B\neq \{1\}$ and let $x$ be a non-trivial element in $A\cap B$. Since $\X$ is an inductive class of groups, Zorn's lemma implies the existence of a maximal $\X$-subgroup $M$ of $G$ containing $A$; we will show that $B\subseteq M$. To this end, let $b\in B$ and note that $B$ is nilpotent by our assumption. Suppose that $B$ is nilpotent of class $k$ and observe that
$$
1=[b,_kx]=[b,_{k-1}x]^{-1}x^{-1}[b,_{k-1}x]x\in M.
$$
But $x$ itself belongs to $M$ so,
$$
1\neq x\in M\cap M^{[b,_{k-1}x]^{-1}}.
$$
Since $M$ is malnormal in $G$, we conclude that $[b,_{k-1}x]\in M$. But then, from
$$
[b,_{k-1}x]=[b,_{k-2}x]^{-1}x^{-1}[b,_{k-2}x]x\in M
$$
we deduce (as before) that $[b,_{k-2}x]\in M$. Repeating this, we eventually see that we must have $[b,x]=b^{-1}x^{-1}bx\in M$. Again, using the malnormality of $M$ together with the fact that $x\in M$, we obtain the desired result that $b\in M$. This shows that $B\subseteq M$ and hence, $\langle A, B \rangle\subseteq M$. Finally, since we are assuming that $\X$ is subgroup-closed, we may conclude that $\langle A,B \rangle$ is an $\X$-group which completes the proof.
\end{proof}

Now, suppose $p$ is an odd prime and consider  $\X=\mathrm{Var}(D_{2p})$, the variety generated by the dihedral group $D_{2p}$. Consider the following copies of the group $D_{2p}$:
$$
A=\langle a_1, a_2:\ a_1^2=a_2^p,\ a_1a_2=a_2^{-1}a_1\rangle,
$$
$$
B=\langle\ b_1, b_2:\ b_1^2=b_2^p,\ b_1b_2=b_2^{-1}b_1\rangle.
$$
Then, consider the amalgamated free product $G=A\ast_{a_1=b_1}B$.

\begin{theorem}
The group $G$ belongs to $\CSX$ but not to $\XT$.
\end{theorem}

\begin{proof}
First, note that the variety $\X$ is locally finite. Consider $A$ and $B$ as subgroups of $G$. We have $A, B\in \X$ and $A\cap B\neq \{1\}$. But, $\langle A, B\rangle =G\not\in \X$. This shows that $G$ does not belong to $\XT$. Before proving  that $G$ is a $\CSX$-group, we show that $A$ and $B$ are malnormal subgroups of $G$. It is enough to prove that $A$ is malnormal since the argument for $B$ is analogous. Suppose $A\cap A^g\neq \{1\}$ for some $g\in G$. Let
$$
g=x_1y_1x_2y_2\cdots x_my_m
$$
be a reduced form of $g$ so that,
$$
x_i\in A\setminus \langle a_1\rangle\quad \mathrm{for}\quad 2\leq i\leq m
$$
and
$$
y_i\in B\setminus \langle b_1\rangle\quad \mathrm{for}\quad 1\leq i\leq m-1.
$$
There is a non-identity element $h\in A$ such that
$$
y_m^{-1}x_m^{-1}\cdots y_1^{-1}x_1^{-1}hx_1y_1\cdots x_my_m\in A.
$$
This shows that $m=1$, and hence, $y_1^{-1}x_1^{-1}hx_1y_1\in A$. We know that
$$
x_1^{-1}hx_1=a_1^ia_2^r\quad \mathrm{for\ some}\quad i=0, 1,\ 0\leq r\leq p-1
$$
and
$$
y_1=b_1^jb_2^s\quad \mathrm{for\ some}\quad j=0, 1,\ 0\leq s\leq p-1.
$$
Thus, we have
\begin{eqnarray*}
y_1^{-1}x_1^{-1}hx_1y_1&=&b_2^{-s}b_1^{-j}a_1^ia_2^rb_1^jb_2^s\\
                       &=&b_2^{-s}a_1^{i-j}a_2^ra_1^Jb_2^s\\
                       &=&b_2^{-s}a_1^ia_2^{-r}b_2^s\\
                       &=&a_1^ib_2^sa_2^{-r}b_2^s,
\end{eqnarray*}
and obviously, this last element has length $3$ except in the following cases:
\begin{enumerate}
\item $s=0$. In this case we have $y_1=b_1^j=a_1^j$ so $g=x_1a_1^j\in A$.
\item $r=0$. In this case $g^{-1}hg=a_1^ib_2^{2s}\in A$ and this implies that $b_2^{2s}=1$. As $s\neq 0$, the odd prime $p$ must divide $2s$ which is impossible since $0<s\leq p-1$.
\end{enumerate}
This shows that $A$ is a malnormal subgroup of $G$; similarly $B$ is a malnormal subgroup of $G$.

Now, we are ready to prove that $G$ belongs to the class $\CSX$. Suppose that $M$ is a maximal $\X$-subgroup of $G$. We claim that $M$ is included either in a conjugate of $A$ or in a conjugate of $B$. Recall that, this is true for the case when $M$ is finite. But, at the moment, we only know that $M$ is locally finite. Set
$$
T=\{ Ag:\ g\in G\}\cup\{ Bg:\ g\in G\},
$$
and define a right action of $G$ on $T$ as follows:
$$
(Ag)x=Agx,\quad (Bg)x=Bgx.
$$
For every finite subset $S\subseteq M$, we have
$$
\langle S\rangle \subseteq A^g\ \mathrm{or}  \hspace{0.1cm}\langle S\rangle\subseteq B^g,
$$
for some $g\in G$. Suppose for example that we have $\langle S\rangle\subseteq A^g$. Then $Ag$ is a fixed point of $\langle S\rangle$ and we denote this by
$$
Ag\in \mathrm{Fix}_T(\langle S\rangle).
$$
Hence, for all finite subsets $S\subseteq M$, the set $\mathrm{Fix}_T(\langle S\rangle)$ is non-empty; it is also finite. To see this, suppose $Ag$ and $Ah$ belong to $\mathrm{Fix}_T(\langle S\rangle)$. Then $A^g\cap B^h\neq \{1\}$, and so, $gh^{-1}\in A$ follows from the malnormality of $A$. Consequently,
$$
|\{ Ah:\ Ah\in \mathrm{Fix}_T(\langle S\rangle)\}|\leq |A|=2p.
$$
A similar argument shows that if some coset $Bh$ belongs to $\mathrm{Fix}_T(\langle S\rangle)$, then the number of such elements is at most $2p$. This means that for any finite subset $S\subseteq M$, we have
$$
0\neq |\mathrm{Fix}_T(\langle S\rangle)|\leq 4p.
$$
Define a set
$$
T_M=\{ \mathrm{Fix}_T(\langle S\rangle): S\subseteq M\setminus \{1\}\ \mathrm{and}\ |S|<\infty\}.
$$
Suppose $S_0$ is a finite subset of $M\setminus \{1\}$ such that $\mathrm{Fix}_T(\langle S_0\rangle)$ is as small as possible. Then, for any other $S$, we have
$$
\mathrm{Fix}_T(\langle S_0\cup S\rangle)\subseteq \mathrm{Fix}_T(\langle S_0\rangle),
$$
and therefore,
$$
\mathrm{Fix}_T(\langle S_0\cup S\rangle)=\mathrm{Fix}_T(\langle S_0\rangle).
$$
But,
$$
\mathrm{Fix}_T(\langle S_0\cup S\rangle)\subseteq \mathrm{Fix}_T(\langle S_0\rangle)\cap \mathrm{Fix}_T(\langle S\rangle)
$$
as well which shows that $\mathrm{Fix}_T(\langle S_0\rangle)\subseteq \mathrm{Fix}_T(\langle S\rangle)$. As a result, the intersection of the elements of $T_M$ is non-empty and this in turn means that $M\subseteq A^g$ for some $g\in G$. Now, by the maximality of $M$, we have $M=A^g$, and hence, $M$ is malnormal, proving that $G$ is $\CSX$.
\end{proof}

Note that the variety $\X$ above is metabelian. For the variety $\mathfrak{A}_2$ of metabelian groups, we establish a property weaker than the $\XT$ property; this is proved below.

\begin{proposition}
Let $\X=\mathfrak{A}_2$ be the variety of metabelian groups.  Let $G$ be a $\CSX$-group. Then, for any  $\X$-subgroups $A$ and $B$, with non-abelian  $A\cap B$, the subgroup $\langle A, B\rangle$ is metabelian.
\end{proposition}

\begin{proof}
Let $A$ and $B$ be metabelian subgroups of $G$ such that $A\cap B$ is not abelian; let $x,y\in A\cap B$ with $[x,y]\neq 1$ and let $M$ be a maximal metabelian subgroup of $G$ containing $A$. We will show that $B\subseteq M$. To this end, let $b\in B$ and note that since $x$, $y$, and $b$ are all elements of the metabelian group $B$,
$$
1=[[x,b],[x,y]]=[x,b]^{-1}[x,y]^{-1}[x,b][x,y]\in M.
$$
Because $[x,y]\in M$, we deduce that
$$
1\neq [x,y]^{-1}\in M\cap M^{[x,b]^{-1}}.
$$
As $M$ is malnormal in $G$, we deduce that $[x,b]=x^{-1}b^{-1}xb\in M$. But $x\in M$ so $1\neq x\in M\cap M^b$; using  malnormality of $M$ once again, we see that $b\in M$. Therefore, $B\subseteq M$ and so, $\langle A,B \rangle$ is metabelian.
\end{proof}

\section{Equational Domains}

As previously noted, the main idea of the present work originates in problems of algebraic geometry over groups. For completeness, we review the necessary background and then show that any inductive class containing all cyclic groups gives rise to a large class of {\em equational domains}.

We use the same notations as in \cite{BMR}, \cite{DMR}, and \cite{Omar-Shah}. Suppose $G$ is a group and $X=\{x_1, x_2, \ldots, x_n\}$ is a set of variables. Let $\mathbb{F}[X]$ be the free group generated by $X$ and $G[X]=G\ast \mathbb{F}[X]$ be the free product of $G$ and $\mathbb{F}[X]$. Every element of $G[X]$ is a group word in the variables $x_1, x_2, \ldots, x_n$ with coefficients from $G$. If $w(x_1, \ldots, x_n)\in G[X]$, then $w(x_1, \ldots, x_n)\approx 1$ is called a group equation. The set
$$
\{ (g_1, \ldots, g_n)\in G^n:\ w(g_1, \ldots, g_n)=1\}
$$
is the solution set of the given equation in $G$. A system of equations with coefficients from $G$ is any set of equations $S\approx 1$, where $S\subseteq G[X]$. The {\em algebraic set} corresponding to this system is the set of all common solutions of all equations in $S\approx 1$ in $G^n$. We denote this algebraic set by $V_G(S)$.

Suppose, for every $S\subseteq G[X]$, there exists a finite subset $S_0\subseteq S$ such that $V_G(S)=V_G(S_0)$. In this case, we say that the group $G$ is {\em equationally Noetherian}. If $G$ is equationally Noetherian, then every algebraic set can be decomposed uniquely as a finite union of {\em irreducible } algebraic sets. The group $G$ is called a {\em domain} if and only if, for every natural number $n$, the union of every two algebraic sets in $G^n$ is again an algebraic set.  There is another definition for the concept of a domain in terms of {\em zero divisors}. An element $x\in G$ is called a {\em zero divisor}, if there exists a non-identity element $y\in G$ such that, for every $g\in G$, we have $[x^g, y]=1$. In \cite{DMR}, it is proved that a $G$ is  a domain if and only if $G$ does not contain any non-trivial zero divisor. It is not hard to see that $G$ is a domain if and only if it satisfies this property: For every non-trivial normal subgroup $K\leq G$,  the centralizer $C_G(K)$ is trivial.  As an example, every non-abelian free group is a domain. This can be generalized to $\mathrm{CSA}$ groups. It is proved  that every $\mathrm{CSA}$ group is a domain (see also \cite{BMR}), a result which is generalized to the case $\CSX$ groups, $\X$ being the variety of nilpotent groups of class at most $k$ (see \cite{Omar-Shah}).  Now, we present a more general result.

\begin{theorem}
Let $\X$ be an inductive class of groups which contains all cyclic groups. If $G$ is a $\CSX$-group that is not an $\X$-group, then $G$ is an equational domain.
\end{theorem}

\begin{proof}
Suppose, for a contradiction, that $G$ contains a non-trivial zero divisor. That is, $G$ contains a pair of non-identity elements $x,y$ such that, for all $g\in G$,
$[x^g,y]=1$. The subgroup $\langle y\rangle$ of $G$ is cyclic, and as such, it is an $\X$-subgroup of $G$ that contains $y$. Since we are assuming that $\X$ is closed under unions of chains, Zorn's lemma implies the existence of a maximal $\X$-subgroup $M$ of $G$ that contains $y$. Observe that for all $g\in G$, we have
$$
1=[x^g,y]=(x^g)^{-1}y^{-1}x^gy\in M;
$$
since $y\in M$ as well, we obtain the following for every $g\in G$
$$
1\neq y^{-1}\in M\cap M^{(x^g)^{-1}}.
$$
Because $M$ is malnormal in $G$, we conclude that $x^g\in M$ for all $g\in G$. In particular, $x\in G$. But then, given any $g\in G$
$$
1\neq x\in M\cap M^{g^{-1}}.
$$
Using the fact that $M$ is malnormal one more time, we conclude that $G\subseteq M$, and hence, $G=M$. This contradicts our assumption that $G$ is not an $\X$-group which completes the proof.
\end{proof}

\section{Free product}

Throughout this section, $\X$ denotes a variety of groups. Using the free product construction, we obtain further $\CSX$-groups; the same idea can be used to produce new examples $\XT$-groups. In what follows, the free product of a family $\{ G_i\}_{i\in I}$ of groups is denoted by $\prod^{\ast}_{i\in I}G_i$. Recall that the free product of any two groups of orders at least three contains the free group $\mathbb{F}_2$. This fact will be used in the proof of the next result.

\begin{theorem}\label{Free}
Suppose $A$ and $B$ are $\CSX$-groups containing no involutions. Then the free product $G=A\ast B$ belongs to $\CSX$.
\end{theorem}

\begin{proof}
Suppose $M$ is a maximal $\X$-subgroup of $G$. By the Kurosh Subgroup Theorem, $M$ has the form
$$
M=\mathbb{F}[X]\ast \left(\prod^{\ast}_{i\in I}K_i^{x_i}\right)\ast\left(\prod^{\ast}_{j\in J}L_j^{y_j}\right),
$$
where $X\subseteq G$, $\mathbb{F}[X]$ is the free group generated by $X$, each $K_i$ is a non-trivial subgroup of $A$, each $L_j$ is a non-trivial subgroup of $B$, and, every $x_i$ and $y_j$ is an element of $G$. We have a few cases to be considered, these are discussed below:
\begin{enumerate}
\item Suppose the above free product decomposition of $M$ contains at least two free factors. Then, the free group $\mathbb{F}_2$ is a subgroup of $M$ and this means that $\mathbb{F}_2\in \X$. Therefore, $\X$ is the variety of all groups. Consequently, $M=G$ which is trivially malnormal.
\item Assume $M=K_i^{x_i}$ or $M=L_j^{y_j}$ for some $K_i\leq A$ or some $L_j\leq B$ and elements $x_i, y_j\in G$. Without loss of generality, we may assume that $M=K_i^{x_i}$. This means that $M\subseteq A^x$ for some $x\in G$. Suppose for some $g\in G$, we have $M\cap M^g\neq \{1\}$. Then, $(M\cap M^g)^{x^{-1}}\neq \{1\}$, and hence,
$$
M^{x^{-1}}\cap M^{gx^{-1}}\neq \{1\}.
$$
We know that $M^{x^{-1}}\subseteq A$ and $M^g\subseteq A^{xg}$ so $M^{gx^{-1}}\subseteq A^{xgx^{-1}}$. Consequently,
$$
A\cap A^{xgx^{-1}}\neq \{1\}.
$$
As $A$ is a free factor of $G$, we must have $xgx^{-1}\in A$. Hence, $g\in A^x$. But, according to our assumption, $A^x$ is a $\CSX$-group. Since $M$ is a maximal $\X$-subgroup of $A^x$, we may conclude that $M$ is malnormal in $A^x$, and as such, the assumption $M\cap M^g\neq 1$ implies that $g\in M$. This proves that $M$ is malnormal in $G$.
\item The last case is the case when $M=\langle z\rangle$ is an infinite cyclic group. If the cyclically reduced length of $z$ is equal to one,
then $z$ belongs to a conjugate of $A$ or a conjugate of $B$. So, we can use a similar argument as in the previous case. Therefore, assume that the cyclically reduced length of $z$ is bigger than one. Let $M\cap M^g\neq \{1\}$ for some $g\in G$. Then, there are two non-zero integers $m$ and $n$ such that $g^{-1}z^ng=z^m$. Comparing the cyclically reduced lengths of the two sides, we conclude that $m=\pm n$. If $m=n$, then $gz^n=z^ng$, and thus, $[g, z^n]=1$. This means that $g$ commutes with all elements of $M$, and so, the subgroup $\langle g, M\rangle$ is abelian. Recall that $M$ is an infinite cyclic group which belongs to $\X$, so every abelian group is an element of $\X$. Now, as $M$ is a maximal $\X$-subgroup of $G$, we must have $M=\langle g, M\rangle$, and accordingly, $g\in M$. It remains to address the case when $m=-n$. In this case, we have $g^{-1}z^ng=z^{-n}$, and hence, $[g^2, z^n]=1$. We know that $g^2\neq 1$ as neither of $A$ nor $B$  contains any involution. Hence, $g$ commutes with $z$, and again, the same argument can be employed to deduce that $g\in M$.
\end{enumerate}
Consequently, $G$ belongs to $\CSX$.
\end{proof}

\section{ $\CSX$ and $\XT$ as Universal Classes }

The classes $\mathrm{CT}$ and $\mathrm{CSA}$  are universal; they can be axiomatize by universal sentences in the first order language of groups (see \cite{Fine}). The same is true for the classes $\XT$ and $\CSX$ when $\X$ is the variety of nilpotent groups of class at  most $k$ (see \cite{Shah}). In this section, we generalize this to a large class of varieties; finitely based varieties $\X$ such that the $2$-generator relatively free group of $\X$ is finitely presented. Note that this includes  every finitely based locally finite variety as well as every virtually nilpotent variety. Recall that if a variety is defined by a finite set of identities
$$
\X=\mathrm{Var}(w_1\approx 1, \ldots, w_m\approx 1),
$$
then it is defined by the single identity
$$
w_1(x_1,\ldots, x_n)w_2(y_1, \ldots, y_n)\cdots \approx 1
$$
where $x_1, \ldots, x_n, y_1, \ldots, y_n, \ldots $ are distinct variables. Accordingly, in what follows, every finitely based variety is assumed to be defined by a single identity. For every variety $\X$, the $2$-generator relatively free element of $\X$ will be denoted by $F_2(\X)$.

\begin{theorem}\label{Universal}
Suppose $\X$ is a finitely based variety such that  $F_2(\X)$ is finitely presented. Then, the class $\XT$ is universal. Further, if $\CSX\subseteq \XT$, then $\CSX$ is universal as well.
\end{theorem}
\begin{proof}
Suppose $\X=\mathrm{Var}(w\approx 1)$. By hypothesis, we have a finite presentation
$$
F_2(\X)=\langle X|R\rangle
$$
with $|X|=2$. Let $w(\mathbb{F}_2)$ denotes the verbal subgroup of the free group $\mathbb{F}_2$ corresponding to the word $w$. Also, suppose $\langle R^{\mathbb{F}_2}\rangle$ is the normal closure of $R$ in $\mathbb{F}_2$. Thus,
$$
F_2(\X)=\frac{\mathbb{F}_2}{w(\mathbb{F}_2)}\cong\frac{\mathbb{F}_2}{\langle R^{\mathbb{F}_2}\rangle},
$$
As both sides of the above isomorphism are relatively free, the subgroup $\langle R^{\mathbb{F}_2}\rangle$ is verbal, and hence,
$$
w(\mathbb{F}_2)=\langle R^{\mathbb{F}_2}\rangle.
$$
Since $R\subseteq w(\mathbb{F}_2)$, there is a finite subset $S$ of $\mathbb{F}_2$ such that $R=w(S)$. Therefore,
$$
w(\mathbb{F}_2)=\langle w(S)^{\mathbb{F}_2}\rangle.
$$
Now, let $A=\langle X\rangle$ be a finitely generated group with $|X|=2$ so that $A=\mathbb{F}_2/K$ for some normal subgroup $K$ of $\mathbb{F}_2$. Then, for some finite subset $S$ of $\mathbb{F}_2$, we have
$$
w(A)=w(\frac{\mathbb{F}_2}{K})=\frac{w(\mathbb{F}_2)K}{K}=\frac{\langle w(S)^{\mathbb{F}_2}\rangle K}{K}\cdot
$$
Accordingly,
$$
w(A)=\langle w(\bar{S})^{\mathbb{F}_2/K}\rangle,
$$
where $\bar{S}=\{ sK:\ s\in S\}$ is finite. As such, $w(A)=\langle w(\bar{S})^A\rangle$. Hence, we proved that  there is a finite subset $S\subseteq \mathbb{F}_2$ (which depends only on $\X$) such that, for every $2$-generated group $A$,  $w(A)=\langle w(\bar{S})^A\rangle$.   This means that in order to verify that a $2$-generated group $A$ belongs to $\X$, it is enough to verify finitely many equalities $w(S)=\{1\}$ inside the group $A$.

Now, we are ready to show that the class $\XT$ is universal. We will use the above result together with Proposition \ref{Cent-1}. Note that, according to what we just proved, there is a finite subset $S(x, y)$  of the free group generated by $x$ and $y$ (which depends only on $\X$) such that  a $2$-generator group $A=\langle a, b\rangle$ belongs to $\X$ if and only if $w(S(a,b))=\{1\}$. Define the first order formula $Q(x,y)$ to be the following:
$$
\bigwedge_{x_1,\ldots,x_n\in S(x,y)}\left(w(x_1,\ldots,x_n)\approx 1\right).
$$
Consequently, in a group $G$, $Q(a,b)$ is true if and only if $\langle a,b\rangle\in \X$.

Now, we use Proposition \ref{Cent-1}: Given a group $G$, the property that for all non-identity elements $x$ of $G$ either $C_G^{\X}(x)=\emptyset$ or $C_G^{\X}(x)\leq G$ translates to the following sentence which refer to  as $\mathrm{Sub}_{\X}$:
$$
\forall x \forall y \forall z (x\neq 1 \land Q(x,y)\land Q(x,z)\longrightarrow Q(x,y^{-1}z)).
$$
Moreover, the property that $C_G^{\X}(x)\in \X$ translates to the following set of sentences, one for each $n\geq 1$ which can be denoted by $\X^n$:
$$
\forall x_1 \forall x_2\ldots\forall x_n \forall x(x\neq 1 \wedge\bigwedge_{i=1}^nQ(x,x_i)\longrightarrow w(x_1,\ldots,x_n)\approx 1).
$$
Hence, the property of being $\XT$ is axiomatized by the above set of universal sentences. In other words, a group $G$ belongs to $\XT$ if and only if it is a model of the set
$$
\{ \mathrm{Sub}_{\X}, \X^n:\ n\geq 1\}.
$$

Next, suppose that $\CSX\subseteq \XT$. According to \ref{CSX}, we know that an $\XT$-group $G$ belongs to $\X$ if and only if
$$
\forall x\neq 1 \forall z (\langle x,x^z \rangle\in \X\longrightarrow \langle x,z\rangle\in \X)
$$
and this translates to ($\mathrm{Mal}_{\X}$):
$$
\forall x\forall z(x\neq 1\wedge Q(x,x^z)\longrightarrow Q(x,z)).
$$
As such, the property of being $\CSX$ is automatized by the set
$$
\{ \mathrm{Sub}_{\X}, \mathrm{Mal}_{\X}, \X^n:\ n\geq 1\}.
$$
\end{proof}

There are many examples of varieties satisfying the assumptions of \ref{Universal} among which are the varieties of nilpotent groups of class at most $k$, every finitely based locally finite variety (equivalently, every locally finite variety with finite axiomatic rank) and in general, every virtually finite nilpotent variety. However, we are not sure about the existence of any more non-trivial examples. An old question of A. Olshanskii about the existence of a non-trivial example of finitely presented relatively free group which is not virtually nilpotent is still unsolved (see Problem 11.73 in  \cite{Kur}). One can use the above theorem to produce infinitely many examples of $\XT$- and $\mathrm{CSX}$-groups.

\begin{corollary}
Suppose $\X$ is a finitely based variety such that the relatively free element $F_2(\X)$ is finitely presented. Then, any ultra-product of $\XT$-groups is $\XT$. If further, $\CSX\subseteq \XT$, then every ultra-product of $\CSX$-groups is $\CSX$.
\end{corollary}

As an example, if we have two $\CSX$-groups $A$ and $B$ without involutions, then for any set $I$ and any ultra-filter $\mathcal{U}$ over $I$, the ultra-power $(A\ast B)^I/\mathcal{U}$ is a new $\CSX$-group. Consequently, we have an infinite supply of new $\CSX$-groups (and hence, equational domains).

\section{Residually $A$-free groups}
It is well-known that the classes of $\mathrm{CT}$- and $\mathrm{CSA}$-groups are closely related to the classes of residually free and fully residually free groups, and therefore play an important role in the study of the universal theory of non-abelian free groups (see \cite{BB}, \cite{Champ}, and \cite{Rem}). In \cite{BB} it is proved that in the presence of the property of being residually free, every $\mathrm{CT}$-group belongs to the class $\mathrm{CSA}$, a fact that is generalized to the case $\X=\mathfrak{N}_k$ in \cite{Shah}. We shall extend all this theory to a wider class of varieties in this final section.

We begin by reviewing some basic concepts. Let $A$ be a group. An $A$-free group is a free product of a set of copies of $A$. Clearly, if $A$ is an infinite cyclic group, then $A$-free groups are the ordinary free groups.

The following two facts can be verified (see \cite{Shah}):
\begin{enumerate}
      \item If $|A|\geq 3$, then $A\ast A\ast A$ embeds in $A\ast A$.
      \item If $|A|\geq 3$, then the free group $\mathbb{F}_2$ embeds in $A\ast A$.
\end{enumerate}
In what follows we shall always assume that $A$ has no involution.   A group $G$ is residually $A$-free if, for every non-identity element $g$ in $G$, there is a homomorphism from $G$ to an $A$-free group $F$ such that $g$ maps to a non-identity element of $F$. We say that $G$ is fully residually $A$-free if, for any finite collection of distinct non-trivial elements $g_1,\ldots ,g_n$ of $G$, there is a homomorphism from $G$ to an $A$-free group $F$ that maps each of $g_1,\ldots ,g_n$ to a non-identity element of $F$.   Observe that every (fully) residually free group is (fully) residually $A$-free whenever $A$ is a group having no involutions.

\begin{theorem}
Suppose $\X$ is a finitely based variety such that the relatively free element $F_2(\X)$ is finitely presented.  Let $A\in\X$ be a non-trivial finitely generated equationally Noetherian group with no involutions and assume that $\CSX\subseteq \XT$.  Then, every fully residually $A$-free group belongs to  $\CSX$.
\end{theorem}

\begin{proof}
Suppose that a finitely generated group $G$ is a fully residually $A$-free. We apply  the unification theorem of \cite{DMR} (Theorem A); since the group $H=A\ast A$ is both finitely generated and equationally Noetherian (see Theorem 9.1 of \cite{Sela}), $G$ embeds in an ultra-product $H^I/\mathcal{U}$ of $H$. Now, $A\in \X\subseteq \CSX$ and $A$ does not contain any involution so $H\in \CSX$. But the class $\CSX$ is universal by \ref{Universal} so, $H^I/\mathcal{U}$ is also a $\CSX$-group, and hence, $G\in \CSX$. If $G$ is not finitely generated, then we may use the fact that the class $\CSX$ is universal, and as such, $G$ is a $\CSX$-group if and only if every finitely generated subgroup of $G$ is so.
\end{proof}

In the remaining part of this section, we will need to use {\em marginal subgroups}; we briefly review this concept. Consider a word $w=w(x_1,\ldots,x_n)$ and a group $G$. The corresponding marginal subgroup $w^{\ast}(G)$ consists of all elements $g\in G$ such that, for all $g_1,\ldots,g_n\in G$ and all $1\leq i\leq n$, we have
$$
w(g_1,\ldots,g_{i-1},gg_i,g_{i+1},\ldots,g_n)=w(g_1,\ldots,g_i,\ldots,g_n).
$$
Two basic but important observations regarding the subgroup $w^{\ast}(G)$ are in order. First, it is easy to see that $w^{\ast}(G)$ is a characteristic subgroup of $G$. Secondly, if $\X$ is a variety of groups, then $G\in \X$ if and only if $w^{\ast}(G)=G$ for every $w\in \idX$.
We define
$$
\X^{\ast}(G)=\bigcap_{w\in\idX}w^{\ast}(G).
$$
This is called the marginal subgroup of $G$ corresponding to the variety $\X$.

\begin{lemma}\label{Marginal-1}
Let $\X$ be a variety and $A$ be a non-trivial element in $\X$ with no involutions. Let $G$ be residually $A$-free, and $N$ be a normal subgroup of $G$ that belongs to the variety $\X$. Then, $N\subseteq \X^{\ast}(G)$.
\end{lemma}

\begin{proof}
Suppose on the contrary that $N\not\subseteq \X^{\ast}(G)$ and let $a\in N\setminus \X^{\ast}(G)$. Then, there is an identity $w\in \idnX$ such that $a\not\in w^{\ast}(G)$. Hence, there is an index $1\leq i\leq n$ and elements $g_1, \ldots, g_n\in G$ such that
$$
w(g_1,\ldots,g_{i-1},ag_i,g_{i+1},\ldots,g_n)\neq w(g_1,\ldots,g_n).
$$
But, $G$ is residually $A$-free so, there is a homomorphism $\alpha:G\to A\ast A$ such that
$$
\alpha(w(g_1,\ldots,g_{i-1},ag_i,g_{i+1},\ldots,g_n))\neq \alpha(w(g_1,\ldots ,g_n)),
$$
which means that
$$
w(\alpha(g_1),\ldots,\alpha(a)\alpha(g_i),\ldots,\alpha(g_n))\neq w(\alpha(g_1),\ldots,\alpha(g_n)).
$$
Note that this implies that $\alpha(a)\neq 1$. Now, suppose $K=\mathrm{Im}(\alpha)\leq A\ast A$ (the image of $\alpha$). We have $K\not\in \X$. By the Kurosh subgroup theorem, we have
$$
K=\mathbb{F}[X]\ast \prod_{i\in I}^{\ast}K_i^{x_i},
$$
where $X\subseteq A\ast A$, $\mathbb{F}[X]$ is the free group generated by $X$, each $K_i$ is a subgroup of $A$, and each $x_i$ belongs to $A\ast A$. We have the following cases:
\begin{enumerate}
\item $K=\mathbb{F}[X]$. As $N\unlhd G$, we have $\alpha(N)\unlhd K$. Also, $\alpha(a)\neq 1$, and since $a\in N$, we have $\alpha(N)\neq 1$. Furthermore, $N\in \X$ implies that $\alpha(N)\in \X$. As a result, $\alpha(N)$ is a free group and there are two possibilities:\\
    i- If $\alpha(N)$ is not abelian, then $\mathbb{F}_2$ embeds in $\alpha(N)$ and hence $\mathbb{F}_2\in \X$. This implies that $\X$ is the variety of all groups and therefore, $\X^{\ast}(G)=G$.\\
    ii- If $\alpha(N)$ is cyclic, then as it is also a normal subgroup of $\mathbb{F}[X]$, we must have $|X|=1$. Hence, $K$ is an infinite cyclic group and $\alpha(N)$ is a non-trivial subgroup of $K$. Thus $K\cong \alpha(N)\in \X$, which is a contradiction.
\item Let $K=B\ast C$ for some non-trivial factors $B$ and $C$. As $A$ does not contain any involution, neither do $B$ and $C$. We have
$$
\{1\}\neq \alpha(N)\unlhd K, \quad \alpha(N)\in \X.
$$
Hence, there are two cases \\
i- $\alpha(N)$ is a subgroup of a conjugate of $B$ or $C$. But if this happens, then $\alpha(N)$ will not be a normal subgroup of $K$.\\
ii- $\alpha(N)$ is a free product of two non-trivial groups. Again, in this case we conclude that $\mathbb{F}_2$ iembeds in $\alpha(N)$, and consequently, $\X$ is the variety of all groups.
\end{enumerate}
This completes the proof.
\end{proof}

\begin{lemma}\label{Marginal-2}
If $\X$ is a variety of groups containing all abelian groups and $G$ is an $\XT$-group that does not belong to $\X$, then $\X^*(G)=\{1\}$.
\end{lemma}

\begin{proof}
Suppose on the contrary that $a\neq 1$ is an element of $\X^{\ast}(G)$. Let $g\in G$ be an arbitrary element. Then, every element of the subgroup $\langle a, g\rangle$ has the form
$$
u(a, g)=a^{\lambda_1}g^{\eta_1}a^{\lambda_2}g^{\eta_2}\cdots a^{\lambda_r}g^{\eta_r}
$$
for some integers $\lambda_1, \eta_1, \ldots, \lambda_r, \eta_r$. As $\X^{\ast}(G)\unlhd G$, we can rewrite $u(a, g)$ as
$$
u(a, g)=g^{\alpha}b, \quad \alpha\in \mathbb{Z},\ b\in \X^{\ast}(G).
$$
Suppose $w\in \idnX$ and $u_1, \ldots, u_n\in \langle a, g\rangle$. For every $i$, we have
$$
u_i=g^{\alpha_i}b_i,\quad \alpha_i\in \mathbb{Z},\ b_i\in \X^{\ast}(G).
$$
Therefore,
$$
w(u_1, \ldots, u_n)=w(g^{\alpha_1}b_1, \ldots, g^{\alpha_n}b_n)=w(g^{\alpha_1}, \ldots, g^{\alpha_n})=1.
$$
Note that the last equality holds because $\X$ contains all abelian groups. Consequently, $\langle a, g\rangle\in \X$ for all $g\in G$. Now, it is enough to use the assumption $G\in \XT$ to conclude that
$$
\langle a, g_1, \ldots, g_m\rangle \in \X,
$$
for every finite set $g_1, \ldots, g_m\in G$. This shows that $G\in \X$, which is a contradiction.
\end{proof}

Now, we are ready to prove that in the presence of the residual $A$-free assumption, every $\XT$-group belongs to $\CSX$.

\begin{theorem}
Let $\X$ be a variety which contains all abelian groups and suppose that $\CSX\subseteq \XT$. Let $A\in \X$ be a group without any involution. Then, every residually $A$-free $\XT$-group belongs to $\CSX$.
\end{theorem}

\begin{proof}
Suppose $G\not\in \CSX$. Then, by \ref{CSX}, there are elements $x, z$ such that $x\neq 1$, $\langle x, x^z\rangle \in \X$, and $\langle x, z\rangle \not\in \X$. Let $G_0=\langle x, z\rangle$. Note that $G_0$ is residually $A$-free and $\XT$. Let $N=\langle x^{G_0}\rangle$. We show that $N\in \X$. Note that
$$
N=\langle u^{-1}xu:\ u\in G_0\rangle.
$$
As $\langle x, x^z\rangle \in \X$, by conjugating with $x$ and $z$, all the subgroups
$$
\langle x, z^{-1}xz\rangle,\ \langle x, x^{-1}z^{-1}xzx\rangle,\ \langle z^{-1}xz, z^{-2}xz^2\rangle
$$
belong to $\X$. Using the assumption of $\XT$, we obtain
$$
\langle x, z^{-1}xz, x^{-1}z^{-1}xzx, z^{-1}xz, z^{-2}xz^2\rangle\in \X.
$$
This means that
$$
\langle u^{-1}xu:\ u\in G_0, \ |u|\leq 2\rangle\in \X,
$$
where $|u|$ denotes the word length of the group $u$ in $G_0$ with respect to the generating set $\{ x, z\}$. Using a similar argument, we see that
$$
\langle u^{-1}xu:\ u\in G_0, \ |u|\leq m\rangle\in \X
$$
for every $m\geq 1$. This implies that $N\in \X$ and hence, by \ref{Marginal-1}, we have $N\subseteq \X^{\ast}(G_0)$. But then, \ref{Marginal-2} implies that $N=\{1\}$ which is a contradiction. Therefore, $G$ belongs to $\CSX$.
\end{proof}

As a special case, we have the following corollary.

\begin{corollary}
Let $\X$ be a variety which contains all abelian groups and suppose that $\CSX\subseteq \XT$.  Then, every residually free $\XT$-group belongs to $\CSX$.
\end{corollary}

We have also the the following result.

\begin{theorem}
Let $\X$ be a variety containing all abelian groups and assume that $\CSX\subseteq\XT$. Let $A\in\X$ be a group with no involutions, and let $G\in \XT\setminus \X$ be residually $A$-free. Then, $G$ is fully residually $A$-free.
\end{theorem}

\begin{proof}
We have $\X^{\ast}(G)=\{1\}$ by Lemma \ref{Marginal-2} so there is no normal $\X$-subgroup of $G$ except the trivial one. We proceed by induction. Suppose for all non-identity $x_1, \ldots, x_{m-1}\in G$ there is a homomorphism $\alpha:G\to A\ast A$ such that
$$
\alpha(x_1)\neq 1, \ldots, \alpha(x_{m-1})\neq 1,
$$
and let $g_1, \ldots, g_m\in G$ be non-identity elements. Suppose for every $x\in G$ we have $g_m^x\in C_{\X}(g_1)$. Note that, according to Proposition \ref{Cent-1}, the generalized centralizer $C_{\X}(g_1)$ is a subgroup, and hence,  $H=\langle g_m^G\rangle\subseteq C_{\X}(g_1)$ is a non-trivial normal $\X$-subgroup of $G$. This contradiction shows that there exists $x\in G$ such that $g_m^x\not\in C_{\X}(g_1)$. This means that
$\langle g_1, g_m^x\rangle\not\in \X$, and as a result, there is an identity $w\in \idnX$ and elements
$$
u_1, \ldots, u_n\in \langle g_1, g_m^x\rangle
$$
such that $w(u_1, \ldots, u_n)\neq 1$. Note that each $u_i$ is a word $u_i=u_i(g_1, g_m^x)$. Now, consider the set
$$
\{ g_2, \ldots, g_{m-1}, w(u_1, \ldots, u_n)\}
$$
which has at most $m-1$ elements. By the induction hypothesis, there is a homomorphism $\alpha: G\to A\ast A$ such that
$$
\alpha(g_2)\neq 1, \ldots, \alpha(g_{m-1})\neq 1, \alpha(w(u_1, \ldots, u_n))\neq 1.
$$
Suppose $\alpha(g_1)=1$ or $\alpha(g_m)=1$. In the first case $\alpha(u_i)$ is a power $\alpha(g_m)$ for all $1\leq i\leq n$ and in the second case $\alpha(u_i)$ is a power $\alpha(g_1)$ for all $1\leq i\leq n$. As a result, in either of these two cases we have
$$
\alpha(w(u_1, \ldots, u_n))=w(\alpha(u_1), \ldots, \alpha(u_n))=1.
$$
This contradiction shows that $\alpha(g_1)\neq 1, \ldots, \alpha(g_m)\neq 1$.
\end{proof}

If we consider the special case where $A$ is an infinite cyclic group, then we obtain the following result.

\begin{corollary}
Let $\X$ be a variety containing all abelian groups and assume that $\CSX\subseteq\XT$. Let  $G\in \XT\setminus \X$ be residually free. Then, $G$ is fully residually free.
\end{corollary}

\end{document}